\documentclass[12pt]{amsart}
\usepackage{amsmath,amssymb,amsbsy,amsfonts,amsthm,latexsym, mathtools,amsopn,amstext,amsxtra,euscript,amscd,mathrsfs,color,bm, cite,multirow,tabularx}
\usepackage{graphicx} % For \resizebox command
\usepackage{todonotes,dirtytalk}   
\usepackage{url}
\usepackage[colorlinks,linkcolor=blue,anchorcolor=blue,citecolor=blue]{hyperref}
\usepackage{color}
\usepackage{comment}
\usepackage{caption}
\usepackage{fullpage}
\allowdisplaybreaks

\mathtoolsset{showonlyrefs}
\begin{document}
\newtheorem{problem}{Problem}
\newtheorem{thm}{Theorem}
\newtheorem{lem}[thm]{Lemma}
\newtheorem{crit}[thm]{Criterion}
\newtheorem{claim}[thm]{Claim}
\newtheorem{cor}[thm]{Corollary}
\newtheorem{prop}[thm]{Proposition}
\newtheorem{definition}{Definition}
\newtheorem{question}[thm]{Question}
\newtheorem{Note}[thm]{Notation}

\theoremstyle{remark} 
\newtheorem*{note}{Note}
\newtheorem{rem}{Remark}

%%%%%%%%%%%%%%%%%%%%%%%%%
% Alphabet calligraphic %
%%%%%%%%%%%%%%%%%%%%%%%%%
\def\cA{{\mathcal A}}
\def\cB{{\mathcal B}}
\def\cC{{\mathcal C}}
\def\cD{{\mathcal D}}
\def\cE{{\mathcal E}}
\def\cF{{\mathcal F}}
\def\cG{{\mathcal G}}
\def\cH{{\mathcal H}}
\def\cI{{\mathcal I}}
\def\cJ{{\mathcal J}}
\def\cK{{\mathcal K}}
\def\cL{{\mathcal L}}
\def\cM{{\mathcal M}}
\def\cN{{\mathcal N}}
\def\cO{{\mathcal O}}
\def\cP{{\mathcal P}}
\def\cQ{{\mathcal Q}}
\def\cR{{\mathcal R}}
\def\cS{{\mathcal S}}
\def\cT{{\mathcal T}}
\def\cU{{\mathcal U}}
\def\cV{{\mathcal V}}
\def\cW{{\mathcal W}}
\def\cX{{\mathcal X}}
\def\cY{{\mathcal Y}}
\def\cZ{{\mathcal Z}}

%%%%%%%%%%%%%%%%%%%%%%%
% Alphabet blackboard %
%%%%%%%%%%%%%%%%%%%%%%%
\def\A{{\mathbb A}}
\def\B{{\mathbb B}}
\def\C{{\mathbb C}}
\def\D{{\mathbb D}}
\def\E{{\mathbb E}}
\def\F{{\mathbb F}}
\def\G{{\mathbb G}}
\def\I{{\mathbb I}}
\def\J{{\mathbb J}}
\def\K{{\mathbb K}}
\def\L{{\mathbb L}}
\def\M{{\mathbb M}}
\def\N{{\mathbb N}}
\def\O{{\mathbb O}}
\def\P{{\mathbb P}}
\def\Q{{\mathbb Q}}
\def\R{{\mathbb R}}
\def\S{{\mathbb S}}
\def\T{{\mathbb T}}
\def\U{{\mathbb U}}
\def\V{{\mathbb V}}
\def\W{{\mathbb W}}
\def\X{{\mathbb X}}
\def\Y{{\mathbb Y}}
\def\Z{{\mathbb Z}}

\def\GL{\mathrm{GL}}
\def\SL{\mathrm{SL}}

\DeclarePairedDelimiter\ceil{\lceil}{\rceil}
\DeclarePairedDelimiter\floor{\lfloor}{\rfloor}

\def\xxx{\vskip5pt\hrule\vskip5pt}

\def\dmod#1{\,\left(\textnormal{mod }{#1}\right)}

\newcommand{\commG}[2][]{\todo[#1,color=green!30]{G:#2}}
\let\svthefootnote\thefootnote
\newcommand\freefootnote[1]{%
  \let\thefootnote\relax%
  \footnotetext{#1}%
  \let\thefootnote\svthefootnote%
}

%%%%%%%%%%%%%%%%%%
%% PAPER BEGINS %%
%%%%%%%%%%%%%%%%%%

\title{An explicit bound on $|\zeta_K(1+it)|$}

\author[G. Pelizzari]{Giulia Pelizzari}\thanks{}
\address{School of Science, The University of New South Wales Canberra, Australia}
\email{g.pelizzari@unsw.edu.au}

\pagenumbering{arabic}

\begin{abstract} Let $K$ be a number field of degree $n$ and discriminant $\Delta$. The estimate $|\zeta_K(1+it)|\le C_K\log{t}$, where $C_K$ is an unspecified constant depending on $K$, is a well--known result in the literature. We obtain an explicit upper bound for $|\zeta_K(1+it)|$, depending only on $n$ and $\Delta$. We also show how improvements on the bound depend on improvements on estimates of the ideal counting function $I_K$ and the residue of $\zeta_K$ at $s=1$.

\end{abstract}

\maketitle

%\freefootnote{\textit{Corresponding author}:}

    %\let\thefootnote\relax
{%
\renewcommand{\thefootnote}{}%
\footnote{\textit{Affiliation}: School of Science, University of New South Wales, Canberra.}%
}
{%
\renewcommand{\thefootnote}{}%
\footnote{\textit{E--mail}: g.pelizzari@unsw.edu.au.}%
}
\setcounter{footnote}{0}
\renewcommand{\thefootnote}{\arabic{footnote}}

\section{Introduction}
Let $K$ be a number field with ring of integers $\mathcal{O}_K$. For each ideal $\mathfrak{a}\subseteq\mathcal{O}_K$, let $\mathcal{N}\mathfrak{a}$ denote its norm.
For $s\in\C$ with $\Re(s)>1$, the Dedekind zeta--function associated to $K$ is $\zeta_K(s)=\sum_{\mathfrak{a}\subseteq\mathcal{O}_k}\mathcal{N}\mathfrak{a}^{-s}.$

The Riemann zeta--function $\zeta$ is the special case of $\zeta_K$ with $K=\Q$. It is then natural to ask whether results concerning $\zeta(s)$ generalise to $\zeta_K(s)$.

Let us set $t>0$. 
In analytic number theory, we are interested in the growth rate of $\zeta(s)$ on the line $\sigma=1$. A well--known bound in the literature is
\begin{equation}\label{zetabound}
    \zeta(1+it)=O(\log{t});
\end{equation} see, for example, \cite{titch}. 
Explicit versions of this bound have several applications in modern research. In particular, they play a central role in deriving sharp estimates for $N(T)$, the zero--counting function of $\zeta$ (see \cite{BellWong}).

%it yields subconvexity estimates in the critical strip via the approximate functional equation together with the Phragmén–Lindelöf principle. These tools play a central role in deriving sharp estimates for $N(T)$, the zero--counting function of $\zeta$. (see \cite{BellWong}).
An analogous result to \eqref{zetabound} for the case of the Dedekind zeta--function is known (see  \cite{LangBook}, \cite{ChaNar}). Given $K$ a number field of degree $n$ and discriminant $\Delta$, then for $t>t_0>0$,
    \begin{equation}
        \label{propOlogt} \zeta_K(1+it)=O_{n,\Delta}(\log{t}).
    \end{equation}
We are not aware of explicit bounds of the form \eqref{propOlogt} for $\zeta_K$ for any $K$; we provide one.
\begin{thm} \label{thm1}
    Let $K$ be a number field of degree $n$ and discriminant $\Delta$ and denote by $\kappa$ the residue of the Dedekind zeta--function $\zeta_K$ at $s=1$. Let $C:=e^{28.2n+5}(n+1)^{\frac{5(n+1)}{2}}|\Delta|^{\frac{1}{n+1}} \log{|\Delta|}^{n}$ and let $t\ge t_0>0$. 
    %Consider a bound for the ideal counting function in the following form:     \begin{align} \label{eq:idealcountingf}    I_K(x):= \sum_{n\le x} a(n) = \kappa x + \vartheta C x^{1-f(n_K)},\end{align} where $C:=C_{\Delta_K,n_K}$ is an explicit constant, $\vartheta$ is a complex number of absolute value $|\vartheta|\le1$, and $0<f(n_K)<1$ is a function of $n_K$. 
    Then,
    \begin{align} \label{eq:result}|\zeta_K(1+it)| & \le \left(\frac{n+1}{2}\right)\kappa\log{t}+2\kappa  + \left(\frac{n+1}{2}\right)\kappa\log{C} - \left(\frac{n+1}{2}\right)\kappa\log{\kappa}
\\&\nonumber +\left(\frac{n+1}{2}\right)C +\frac{\sqrt{1+t_0^2}}{t_0}\kappa + \left(\frac{n+1}{2}\right)\left(1+\log2+\frac{1}{2t_0}\right)\kappa.\end{align}
\end{thm}

This estimate could be used for similar future applications (one could improve on \cite{HanShenWong} using the same techniques as in \cite{BellWong}, \cite{Tru14}).

To obtain a fully explicit result\footnote{Here we mean that the result depends only on $n$ and $\Delta$.}, it is also necessary to take into account the most recent estimates of $\kappa$ (see Section \ref{sec2}). In doing so, we obtain the following corollary.

\begin{cor} \label{cor2}
    Let $K$ be a number field of degree $n$ and discriminant $\Delta$ and denote with $\kappa$ the residue of the Dedekind zeta--function $\zeta_K$ at $s=1$. For $t\ge 1$ we get:
    \begin{align}\label{corresult}
|\zeta_K(1+it)| &\le \left(\frac{e \log |\Delta|}{2(n-1)}\right)^{n-1}\left(\frac{n+1}{2}\right)\log{t} +An^{an}|\Delta|^{\frac{1}{n+1}}\log|\Delta|^n, 
\end{align}
where
$$a=10, \quad\text{ and } \quad A=10^{60}.$$
\end{cor}

\begin{rem}
    The constant $A$ depends on the exponent $a$. We choose $10$, although any exponent greater than $2.5$ can be used depending on the range of $n$; see Table \ref{table1} for examples. It should be noted that $a=49$ is the first integer value for which $A<10$.
\end{rem}

\section{Relevant background}\label{sec2}

The first explicit elementary estimate on the $1$--line for $\zeta$ is due to Landau \cite{Land1line}. It was then improved by Backlund \cite{Back}, for $t>50$:
\begin{equation}\label{boundform}
 |\zeta(1+it)|\le C \log{t},
\end{equation} with $C=1$.
Later on Ford \cite{Ford} provided an improvement, building on the work of Richert \cite{Rich}%using the Vinogradov--Korobov estimate for exponential sums
. Backlund's and Ford's results were later improved by Trudgian \cite{Tru}. %Patel \cite{Patel} later gave improved bounds.
The sharpest published explicit results available today are due to Patel \cite{Patel} for $t\ge3$ 
%    $$|\zeta(1+it)|\le \min{\left(\log{t}, \frac{1}{2}\log{t}+1.93,\frac{1}{5}\log{t}+44.02\right)},$$ 
and to Hiary, Leong and Yang \cite{1line} in a specific range. In \cite{christensen2026} Christensen and Pratt provide an improvement for small $t$. Regarding the form \eqref{boundform}, the sharpest bound for $t\ge e$ is due to Qingyi and Teo \cite{qingyi}. %, where $C=0.6443$.

Some progress has also been made in the case of Dirichlet $L$--functions, generalising work on $\zeta$. 
Dudek \cite{Dudek} provides an explicit result for all $t>0$ for moduli $q\ge1$ of non--principal Dirichlet characters $\chi$: $$|L(1+it,\chi)|<\log{\left(t+\frac{14}{5}\right)}+\log{q}+1.$$

It is then natural to ask whether one can generalise known results to find an explicit bound for $\zeta_K$. Given $K$ of degree $n$ and discriminant $\Delta$, it is well--known that $$\zeta_K(s)=\sum_{m=1}^\infty\frac{a(m)}{m^s},$$ where $a(m):=\#\{\mathfrak a\subseteq \mathcal{O}_K: \: \mathcal{N}\mathfrak a = m\}$.
To establish Theorem \ref{thm1}, we need an estimate of the ideal counting function $I_K(x):= \sum_{m\le x} a(m)$.

The earliest bound is attributed to Weber \cite{Web}:
    \begin{equation*} \label{weberbound}
        I_K(x) = \kappa x+O_{n,\Delta}\left(x^{1-\frac{1}{n}}\right).
    \end{equation*}
Fehnker \cite{Fehn} used the geometry of numbers to make it explicit.

%Various other mathematicians worked on the error term in \eqref{weberbound}.

Landau \cite{Land} improved on the exponent from \eqref{weberbound} in 1949:
    \begin{equation} \label{landaubound}
        I_K(x) = \kappa x+O_{n,\Delta}\left(x^{1-\frac{2}{n+1}}\right).
    \end{equation}

This estimate was first made explicit by Sunley \cite{Sunley}.

\begin{prop}[Sunley, 1971] For $x>0$ and $n\ge 2$, we have \begin{align} \label{sunleybound}
|I_K(x)-\kappa x| &\le e^{28.2n+5}(n+1)^{\frac{5(n+1)}{2}}|\Delta|^{\frac{1}{n+1}} \log{|\Delta|}^{n}x^{1-\frac{2}{n+1}}.
\end{align}
\end{prop}

This explicit bound was later improved by Lee \cite{LeeE} by splitting the values of $n$ in two different ranges.
In 1993, Nowak \cite{Now} provided a different formulation of \eqref{landaubound} for $n\ge 3$:
\begin{equation*}
I_K(x)= \kappa x + O_{n,\Delta}(x^{1-\alpha(n)}(\log{x})^{\beta(n)}),
\end{equation*}
where $\alpha(n)$ and $\beta(n)$ are explicit non--zero constants depending on $n$.
Lao \cite{Lao} improved the exponent of \eqref{landaubound} to $1-3/(n+6)+\varepsilon.$
%: \begin{equation}\label{laobound}    I_K(x) = \kappa x+O_{n,\Delta}\left(x^{1-\frac{3}{n+6}+\varepsilon}\right).\end{equation}
We are not aware of any explicit version of these bounds.

Sharper estimates are available under additional restrictions on the number field. %Implicit results are due to Paul and Sankaranarayanan \cite{PaulSan} for cyclotomic fields and to Languasco, Lunia and Moree \cite{Lang} for the abelian case with $n\ge 4$. 
Explicit results are known for quadratic fields (Huxley and Watt \cite{HuWa}), imaginary quadratic fields (Ramar\'e \cite{Ram}), and cubic fields (M\"uller \cite{Mul}); see also Fomenko \cite{Fom} for specific degree $8$ and $16$ fields and Bordell\`es \cite{Bord1} for a weighted version.

\begin{rem}
    In Theorem \ref{thm1}, we use \eqref{sunleybound} to obtain the result. A sharper bound would follow from Lee's estimate \cite{LeeE}. However, this estimate depends on the range, and a more general concise formulation is preferable for applications. Moreover, the proof can be readily adapted to other explicit bounds if needed.
\end{rem}

For the proof of Corollary \ref{cor2}, one needs to estimate $\kappa:=\mathrm{Res}_{s=1}\zeta_K(s)$.
Siegel \cite{Sieg} was the first to provide such an upper bound, and improvements were then given by Lavrik and Egorov \cite{LavEgo}, and by Qu\^eme \cite{Que1}, \cite{Que2}. A further refinement is due to Louboutin \cite{Lou2000}.
\begin{prop}[Louboutin, 2000] Let $K$ be a number field of degree $n>1$ and discriminant $\Delta$. Then,
    \begin{equation} \label{residuebound2000}
        \mathrm{Res}_{s=1}\zeta_K(s) \le \left(\frac{e \log{|\Delta}|}{2(n-1)}\right)^{n-1}.
    \end{equation}
\end{prop}
A sharper implicit result was obtained by Li \cite{Xli} and improved by Louboutin \cite{Loub2015} for $n\ge 5$.
% \begin{equation} \label{residuebounds2015}     \mathrm{Res}_{s=1}\zeta_K(s) \le \frac{3}{\sqrt{n}}\left(B\frac{ce^{\gamma+\sqrt{\frac{6}{n}}} \log{|\Delta}|}{n}\right)^{n-1}, \end{equation} 
% where $c=\frac{(1-1/\sqrt{5})}{2}$, $B=1+O(\log{\rho_K)}$, $\gamma$ denotes the Euler--Mascheroni constant and the value of $\rho_K$ can be computed explicitly if $K$ is known.
%The sharpest bound available is obtained combining \eqref{residuebound2000} and \eqref{residuebounds2015}.
Sharper explicit bounds are available when restricting $K$ to specific classes of number fields; see, for example, \cite{Loub2015} for the abelian case.

%\begin{rem} As before, in Corollary \ref{cor2} we chose to fix $t_0=1$ for simplicity and usefulness of the expression. The reader can adapt the method for a sharper result. \end{rem} 

\section{Proof of Theorem \ref{thm1} and Corollary \ref{cor2}}

 \begin{proof}[Proof of Theorem \ref{thm1}]

The proof follows \cite[Theorem 3.5]{titch}.
Let $\sigma >1$ and
\begin{equation*}
    \zeta_K(s) - \sum_{m\le x} \frac{a_K(m)}{m^s} = \sum_{m > x} \frac{a_K(m)}{m^s}.
\end{equation*}
By partial summation,
\begin{equation} \label{eq:partialsumm1}
    \zeta_K(s) - \sum_{m\le x} \frac{a_K(m)}{m^s} = - \frac{I_K(x)}{x^s} - s\int^\infty_x \frac{I_K(u)}{u^{s+1}} du.
\end{equation}
Let $\vartheta$ be a complex number of absolute value $|\vartheta|\le 1$ and let us restate  \begin{equation}\label{eq:idealcountingf}
    C:=e^{28.2n+5}(n+1)^{\frac{5(n+1)}{2}}|\Delta|^{\frac{1}{n+1}} \log{|\Delta|}^{n};
\end{equation}
one can then rewrite \eqref{sunleybound} as \begin{equation}\label{sunleyboundC}
    I_K(x) = \kappa x + \vartheta C x^{1-\frac{2}{n+1}}.
\end{equation}

From \eqref{sunleyboundC}, one gets 
$$ \left|\frac{I_k(x)}{x^s}\right| \le x^{1-\sigma}\kappa+ C x^{1-\frac{2}{n+1}-\sigma}.$$

Let us substitute the bound \eqref{sunleyboundC} in the integral from \eqref{eq:partialsumm1}. We then split it into two convergent integrals as follows:
\begin{align*}
    \left|-s\int^\infty_x \frac{I_K(u)}{u^{s+1}} du \right|&\le \left|-s\int^\infty_x \frac{\kappa u}{u^{s+1}} du\right|+ 
    |s|\int^\infty_x \frac{ C u^{1-\frac{2}{n+1}}}{u^{\sigma+1}} du \\&\le \frac{x^{1-\sigma}\sqrt{\sigma^2+t^2}}{|t|}\kappa + 
    |s|\int^\infty_x \frac{ C u^{1-\frac{2}{n+1}}}{u^{\sigma+1}} du .
   \end{align*}
Replacing the corresponding terms in \eqref{eq:partialsumm1} with these results yields
\begin{align} \label{eq:halfdone}
    |\zeta_K(s)| &\le \left|\sum_{m\le x} \frac{a_K(m)}{m^s}\right|+\frac{x^{1-\sigma}(\sqrt{\sigma^2+t^2}+|t|)}{|t|}\kappa+  C x^{1-\frac{2}{n+1}-\sigma} + 
    |s|\int^\infty_x \frac{  C u^{1-\frac{2}{n+1}}}{u^{\sigma+1}} du.
\end{align}

Since the integral in \eqref{eq:halfdone} converges for $\sigma>1-2/(n+1)>0$, the expression holds in this region by analytic continuation. Now, let $t>0$. 
Let us estimate the last term in \eqref{eq:halfdone} above:
\begin{align} \label{eq5}
    |s|\int^\infty_x \frac{ C u^{1-\frac{2}{n+1}}}{u^{\sigma+1}} du &\le |s|\int^\infty_x \frac{C}{u^{\sigma+\frac{2}{n+1}}} du\\& \nonumber =\frac{C|s|}{(\sigma-1+\frac{2}{n+1})x^{\sigma-1+\frac{2}{n+1}}}.
\end{align}
%Then\begin{equation*}    |\zeta_K(s)| \le \left|\sum_{n\le x} \frac{a_K(n)}{n^s}\right| +  \frac{x^{1-\sigma}}{|t|}\kappa+ C x^{1-\frac{2}{n_K+1}-\sigma}+\frac{C|s|}{(\sigma-1+\frac{2}{n_K+1})x^{\sigma-1+\frac{2}{n_K+1}}}.\end{equation*}

Now fix $\sigma =1$. 
Let us estimate the sum in \eqref{eq:halfdone} using the triangle inequality and partial summation as follows:
\begin{align*}
    \left|\sum_{m\le x} \frac{a_K(m)}{m^{1+it}}\right| & \le \sum_{m\le x} \frac{a_K(m)}{m}\le \frac{I_K(x)}{x} + \int^x_1 \frac{I_K(u)}{u^{2}} du \\&\le \kappa + C x^{-\frac{2}{n+1}} + \int^x_1 \frac{\kappa}{u} du+ \int^x_1 \frac{ C u^{1-\frac{2}{n+1}}}{u^{2}} du \\&= \kappa +  C x^{-\frac{2}{n+1}}+\kappa\log{x}  - \frac{(n+1) }{2}C x^{-\frac{2}{n+1}}+  \frac{(n+1) }{2}C.
\end{align*}   

We then have
\begin{align} \label{eq6}
    |\zeta_K(1+it)|& \le 2\kappa + 2C x^{-\frac{2}{n+1}}+\kappa\log{x} - \frac{(n+1)}{2}C x^{-\frac{2}{n+1}}\\& \nonumber +\frac{(n+1)}{2}C+  \frac{\sqrt{1+t^2}}{t}\kappa+\frac{(n+1)C\sqrt{1+t^2}}{2x^{\frac{2}{n+1}}}.
\end{align}
Since $\sqrt{1+y^2}/y$ is a decreasing function for $y>0$, for $t\ge t_0>0$ one gets that $$\frac{\sqrt{1+t^2}}{t}\kappa \le \frac{\sqrt{1+t_0^2}}{t_0}\kappa.$$

Finally, by using calculus, one finds that the optimal value for $x$ is \begin{equation}\label{optvalue}
x=\Bigg(\frac{C}{\kappa}\left(\frac{4}{ (n+1)}+\sqrt{1+t^2}-1\right)\Bigg)^\frac{n+1}{2}.
\end{equation}
Then, by substituting \eqref{optvalue} in \eqref{eq6}, one obtains 
\begin{align} \label{eqottobre} 
|\zeta_K(1+it)| & \le 2\kappa  + \left(\frac{n+1}{2}\right)\kappa\log{C} - \left(\frac{n+1}{2}\right)\kappa\log{\kappa} \\& \nonumber+\left(\frac{n+1}{2}\right)\kappa\log\left(\frac{4}{n+1}+\sqrt{1+t^2}-1\right) 
\\&\nonumber +\frac{(n+1)}{2}C +\frac{\sqrt{1+t_0^2}}{t_0}\kappa + \frac{n+1}{2}\kappa.\end{align}

For $n\geq2$ and for $y> 0$ we have that $$\frac{4}{n+1}-1\leq \frac{1}{3}<1, \quad  \sqrt{1+y^2}\leq 2y, \quad \log{(y+1)}\le \log{y}+\frac{1}{y}.$$
Hence,
\begin{equation} \label{eqlogt}
    \left(\frac{n+1}{2}\right)\kappa\log\left(\frac{4}{n+1}+\sqrt{1+t^2}-1\right) \leq \left(\frac{n+1}{2}\right)\kappa\left(\log{t}+\log2+\frac{1}{2t_0}\right).
\end{equation}
Substituting \eqref{eqlogt} in \eqref{eqottobre} and simplifying yields \eqref{eq:result}. This completes the proof.
 \end{proof}

%\begin{rem}     Taking the limit as $t\to \infty$ in \eqref{eq:result}, we immediately get \eqref{propOlogt} as a corollary. \end{rem}

One can then prove Corollary \ref{cor2}.

\begin{proof}[Proof of Corollary \ref{cor2}.]
From now on, we shall write $\Delta$ for the absolute value of the discriminant of $K$. 
    Let us start by substituting \eqref{eq:idealcountingf} into \eqref{eq:result} and fixing $t_0=1$. We obtain

    \begin{align}\label{monstereq}
        |\zeta_K(1+it)|\le T_1+T_2+T_3+T_4+T_5+T_6+T_7,
    \end{align}
    where
%     \left(\frac{n+1}{2}\right)\kappa\log{C} - \left(\frac{n+1}{2}\right)\kappa\log{\kappa}
%\\&\nonumber +\frac{(n+1)}{2}C +\frac{\sqrt{1+t_0^2}}{t_0}\kappa + \frac{n+1}{2}\left(1+\frac{1}{t_0}\right)\kappa.
    \begin{align*}
    T_1:=\left(\frac{n+1}{2}\right)\kappa\log{t}, \quad \quad T_2:=2\kappa,
    \end{align*}
    \begin{align*}
    T_3:=\left(\frac{n+1}{2}\right)\kappa\log\left(e^{28.2n+5}(n+1)^{\frac{5(n+1)}{2}}\Delta^\frac{1}{n+1}\log{\Delta}^{n} \right),
    \end{align*}
    \begin{align*}
        T_4:=-\left(\frac{n+1}{2}\right)\kappa\log\kappa, \quad \quad T_5:=\left(\frac{n+1}{2}\right)e^{28.2n+5}(n+1)^{\frac{5(n+1)}{2}}\Delta^\frac{1}{n+1}\log{\Delta}^{n},
    \end{align*}

    \begin{align*}
        T_6:=\sqrt{2}\kappa, \quad \quad T_7:=\left(\frac{n+1}{2}\right)2.1932\kappa.
    \end{align*}

The most significant terms are $T_1$ and $T_5$. 

By Stickelberger's criterion and Minkowski's bound on the discriminant of a number field, we know that $\Delta>3$ (see \cite{Neuk}). We also note that for $n\geq n_0$, we can compute \begin{equation} \label{f}
    f:=f(n_0)\leq 1/e \quad \text{such that} \quad \log\log n\leq f\log n.
\end{equation}
For $n\geq n_0$ we can then bound $T_3$ as follows:
\begin{align}\label{newT3}
T_3&:=\left(\frac{n+1}{2}\right)\kappa\left(28.2n+5+\frac{5(n+1)}{2}\log(n+1)+\frac{1}{n+1}\log\Delta+n\log\log{\Delta} \right)
\\& \nonumber \leq \left(\frac{n+1}{2}\right)\kappa\left(28.2n+5+\frac{5(n+1)}{2}\log(n+1)+\left(\frac{1}{n+1}+fn\right)\log\Delta\right).\end{align}

%Let us set $K:= \left(\frac{e \log \Delta}{2(n-1)}\right)^{n-1}\frac{n+1}{2}$.
%Since for $x\geq0$, $\max (-x \log x)=1/e$, we get \begin{align} \label{newT4}T_4 &\le \left(\frac{n+1}{2e}\right).\end{align}

%Let us now define $$\mathcal{K}:=\left(\frac{e \log{\Delta}}{2(n-1)}\right)^{n-1} \frac{n+1}{2}.$$

%Now, by substituting \eqref{newT3} and \eqref{newT4} in \eqref{monstereq}, bounding the value of $\kappa$ by \eqref{residuebound2000}, factoring $\mathcal{K}$ out and simplifying, we get

We observe that the right hand side of \eqref{monstereq} can be written as
\begin{equation*}
\left(\frac{n+1}{2}\right)\kappa\log{t}+F_{n,\Delta}(\kappa),
\end{equation*}
where $F$ is increasing in $\kappa$. We can then substitute \eqref{newT3} and \eqref{residuebound2000} in \eqref{monstereq} to obtain
\begin{align} \label{eqwithsubs1}
    |\zeta_K(1+it)| &\leq \left(\frac{e\log{\Delta}}{2(n-1)}\right)^{n-1} \left(\frac{n+1}{2}\right)\log{t} + 2\left(\frac{e \log{\Delta}}{2(n-1)}\right)^{n-1}
    \\&+\left(\frac{n+1}{2}\right)\left(\frac{e \log{\Delta}}{2(n-1)}\right)^{n-1} \Big(28.2n+5+\frac{5(n+1)}{2}\log{(n+1)} \nonumber
    \\&+\left(\frac{1}{n+1}+nf\right)\log\Delta\Big) - \left(\frac{n+1}{2}\right)\left(\frac{e \log{\Delta}}{2(n-1)}\right)^{n-1}\times \nonumber
    \\& \times \log{\left(\left(\frac{e \log{\Delta}}{2(n-1)}\right)^{n-1}\right)} +\left(\frac{n+1}{2}\right)e^{28.2n+5}(n+1)^{\frac{5(n+1)}{2}}\Delta^{\frac{1}{n+1}}\log\Delta^n \nonumber
    \\&+\sqrt{2}\left(\frac{e \log{\Delta}}{2(n-1)}\right)^{n-1}+\left(\frac{n+1}{2}\right)2.1932\left(\frac{e \log{\Delta}}{2(n-1)}\right)^{n-1}. \nonumber
\end{align}

Now we observe that, for $n\geq n_0$, we can compute $c:=c(n_0)\leq 3/2$ such that
\begin{equation}\label{c}
 (n+1)^\frac{5(n+1)}{2}(n+1)\leq (cn)^\frac{5n}{2}(n+1)^\frac{7}{2}.
\end{equation}
 
First, we wish to solve the logarithm in \eqref{eqwithsubs1} and substitute \eqref{f} and \eqref{c}; then, we want to factor out $\Delta^\frac{1}{n+1}\log\Delta^n$. We thus obtain

\begin{align} \label{eqwithsubs2}
    |\zeta_K(1+it)| &\leq \left(\frac{e\log{\Delta}}{2(n-1)}\right)^{n-1} \left(\frac{n+1}{2}\right)\log{t} + \Delta^\frac{1}{n+1}\log\Delta^n\Bigg[3.5\left(\frac{e \log{\Delta}}{2(n-1)}\right)^{n-1}\frac{1}{\Delta^\frac{1}{n+1}\log\Delta^n}
    \\&+\left(\frac{n+1}{2}\right)\left(\frac{e \log{\Delta}}{2(n-1)}\right)^{n-1} \frac{1}{\Delta^\frac{1}{n+1}\log\Delta^n}\Big(28.2n+5+\frac{5(n+1)}{2}\log{(n+1)} \nonumber
    \\&+\left(\frac{1}{n+1}+nf\right)\log\Delta\Big) - \left(\frac{n+1}{2}\right)\left(\frac{e \log{\Delta}}{2(n-1)}\right)^{n-1}\frac{1}{\Delta^\frac{1}{n+1}\log\Delta^n}\times \nonumber
    \\&\times (n-1)(1+f\log\Delta-\log2-\log(n-1))+\frac{c^{\frac{5n}{2}}}{2}(n+1)^\frac{7}{2}n^\frac{5n}{2}e^{28.2+5}\nonumber
    \\& +\left(\frac{e\log{\Delta}}{2(n-1)}\right)^{n-1} \left(\frac{n+1}{2}\right)\frac{2.1932}{\Delta^\frac{1}{n+1}\log\Delta^n}\Bigg]. \nonumber
\end{align}

From $\Delta>3$ we know that, for $n\geq 2$, $\Delta^\frac{1}{n+1}\log\Delta>1/1.09$ and $\Delta^\frac{1}{n+1}>1$.

Now, by computation, we bound the expression in square brackets in \eqref{eqwithsubs2} with a quantity of the form
\begin{align} \label{an}
 An^{an}.
\end{align}
We also look for an optimal value $n_0$, such that combining the results for $n\leq n_0$ and $n>n_0$ gives us a sharper bound. We then obtain the following table.
\begin{table}[h]
    \centering
    \begin{tabular}{|p{30mm}|c|p{60mm}|c|}
    \hline
        $A$ & $a$ \\ \hline  
        \rule[-1ex]{0pt}{4ex}%$10^{3048278}$  &  $4$ \\
        $10^{1780}$           & $6$ \\
        $10^{159}$                &  $8$ \\
        $10^{60} $            & $10$ \\
        $10^{19}$           & $20$ \\
        $8.31$                &  $49$ \\ \hline
    \end{tabular}
    \captionsetup{skip=2pt}
    \caption{Value of $A$ for some choices of $a$.}
    \label{table1}
\end{table}

Substituting \eqref{an} in \eqref{eqwithsubs2} for $a=10$ yields \eqref{corresult}. This concludes the proof.
\end{proof}

%\begin{rem}     For big values of $a$, such as $20$ and $49$, we tried to bound $\Delta^{1/(n+1)}>3^{1/(n_0+1)}$, since the max of the function of $n$ in the square parenthesis occurs at $n=2$. However, this didn't produce any improvement on the value of the constant $A$.\end{rem}

\section*{Acknowledgements}
The author is grateful to Tim Trudgian for suggesting the problem and providing valuable feedback. The author thanks Olivier Bordell\`es for helpful references on estimates of $I_K$ and Matthew Fernando for insightful discussions. 
This research was supported by the Commonwealth through an Australian Government Research Training Program Scholarship [DOI: https://doi.org/10.82133/C42F-K220].
\bibliographystyle{plain}
\bibliography{refs}

@book {titch,
    AUTHOR = {Titchmarsh, E. C.},
     TITLE = {The theory of the {R}iemann zeta-function, second edition},
      NOTE = {Revised by D. R. Heath-Brown},
 PUBLISHER = {The Clarendon Press, Oxford University Press, New York},
      YEAR = {1986},
   }

@article {ChaNar,
    AUTHOR = {Chandrasekharan, K. and Narasimhan, R.},
     TITLE = {The approximate functional equation for a class of
              zeta-functions},
   JOURNAL = {Math. Ann.},
  FJOURNAL = {Mathematische Annalen},
    VOLUME = {152},
      YEAR = {1963},
     PAGES = {30--64},
      ISSN = {0025-5831,1432-1807},
   MRCLASS = {10.41},
  MRNUMBER = {153643},
MRREVIEWER = {S.\ Chowla},
       DOI = {10.1007/BF01343729},
       URL = {https://doi-org.wwwproxy1.library.unsw.edu.au/10.1007/BF01343729},
}

@article {LeeE,
    AUTHOR = {Lee, E. S.},
     TITLE = {On the number of integral ideals in a number field},
   JOURNAL = {J. Math. Anal. Appl.},
  FJOURNAL = {Journal of Mathematical Analysis and Applications},
    VOLUME = {517},
      YEAR = {2023},
    NUMBER = {1},
     PAGES = {Paper No. 126585, 25},
      ISSN = {0022-247X,1096-0813},
   MRCLASS = {11R42},
  MRNUMBER = {4469611},
MRREVIEWER = {Krishnarjun\ Krishnamoorthy},
       DOI = {10.1016/j.jmaa.2022.126585},
       URL = {https://doi-org.wwwproxy1.library.unsw.edu.au/10.1016/j.jmaa.2022.126585},
}

@article {Bord1,
    AUTHOR = {Bordell\`es, O.},
     TITLE = {On the ideal theorem for number fields},
   JOURNAL = {Funct. Approx. Comment. Math.},
  FJOURNAL = {Uniwersytet im. Adama Mickiewicza w Poznaniu. Wydzia\l\
              Matematyki i Informatyki. Functiones et Approximatio
              Commentarii Mathematici},
    VOLUME = {53},
      YEAR = {2015},
    NUMBER = {1},
     PAGES = {31--45},
      ISSN = {0208-6573,2080-9433},
   MRCLASS = {11N37 (11L07 11R42)},
  MRNUMBER = {3402771},
MRREVIEWER = {Volker\ Ziegler},
       DOI = {10.7169/facm/2015.53.1.3},
       URL = {https://doi-org.wwwproxy1.library.unsw.edu.au/10.7169/facm/2015.53.1.3},
}

@article {Ram,
    AUTHOR = {Ramar\'e, O.},
     TITLE = {Explicit count of integral ideals of an imaginary quadratic
              field},
   JOURNAL = {Ann. Math. Qu\'e.},
  FJOURNAL = {Annales Math\'ematiques du Qu\'ebec},
    VOLUME = {49},
      YEAR = {2025},
    NUMBER = {2},
     PAGES = {463--475},
      ISSN = {2195-4755,2195-4763},
   MRCLASS = {11N37 (11N45 11R11)},
  MRNUMBER = {4967772},
MRREVIEWER = {W.\ Narkiewicz},
       DOI = {10.1007/s40316-025-00243-0},
       URL = {https://doi-org.wwwproxy1.library.unsw.edu.au/10.1007/s40316-025-00243-0},
}

@article {Dudek,
    AUTHOR = {Dudek, A. W.},
     TITLE = {An explicit result for {$|L(1+it,\chi)|$}},
   JOURNAL = {Funct. Approx. Comment. Math.},
  FJOURNAL = {Uniwersytet im. Adama Mickiewicza w Poznaniu. Wydzia\l\
              Matematyki i Informatyki. Functiones et Approximatio
              Commentarii Mathematici},
    VOLUME = {53},
      YEAR = {2015},
    NUMBER = {1},
     PAGES = {23--29},
      ISSN = {0208-6573,2080-9433},
   MRCLASS = {11M06},
  MRNUMBER = {3402770},
MRREVIEWER = {Marc\ Munsch},
       DOI = {10.7169/facm/2015.53.1.2},
       URL = {https://doi-org.wwwproxy1.library.unsw.edu.au/10.7169/facm/2015.53.1.2},
}

@book {Land,
    AUTHOR = {Landau, E.},
     TITLE = {Einf\"uhrung in die elementare und analytische {T}heorie der algebraischen {Z}ahlen und der {I}deale},
 PUBLISHER = {Chelsea Publishing Co., New York},
      YEAR = {1949},
     PAGES = {vii+147},
   MRCLASS = {10.0X},
  MRNUMBER = {31002},
}

@book{Web,
    author = {Weber, H.},
    title = {Lehrbuch der Algebra},
    volume = {II},
    publisher = {Vieweg, Braunschweig},
    year = {1896}
}

@book {Sunley,
    author = {Sunley, J. E. S.},
     title = {On the class numbers of totally imaginary quadratic extensions of totally real fields},
      note = {Thesis (Ph.D.)--University of Maryland, College Park},
 publisher = {ProQuest LLC, Ann Arbor, MI},
      year = {1971},
     pages = {111},
   MRCLASS = {99-05},
  MRNUMBER = {2621002},
       URL =
              {http://gateway.proquest.com.wwwproxy1.library.unsw.edu.au/openurl?url_ver=Z39.88-2004&rft_val_fmt=info:ofi/fmt:kev:mtx:dissertation&res_dat=xri:pqdiss&rft_dat=xri:pqdiss:7201650},
}

@article{1line,
author = {G. A. Hiary and N. Leong and A. Yang},
title = {{Explicit bounds for the Riemann zeta-function on the $1$-line}},
volume = {73},
journal = {Funct. Approx. Comment. Math.},
number = {1},
publisher = {Adam Mickiewicz University, Faculty of Mathematics and Computer Science},
pages = {53 -- 89},
year = {2025},
doi = {10.7169/facm/240711-13-8},
URL = {https://doi.org/10.7169/facm/240711-13-8}
}

@article {Patel,
    AUTHOR = {Patel, D.},
     TITLE = {An explicit upper bound for {$|\zeta (1+it) |$}},
   JOURNAL = {Indag. Math. (N.S.)},
  FJOURNAL = {Koninklijke Nederlandse Akademie van Wetenschappen.
              Indagationes Mathematicae. New Series},
    VOLUME = {33},
      YEAR = {2022},
    NUMBER = {5},
     PAGES = {1012--1032},
      ISSN = {0019-3577,1872-6100},
   MRCLASS = {11M06},
  MRNUMBER = {4470692},
MRREVIEWER = {Athanasios\ Sourmelidis},
       DOI = {10.1016/j.indag.2022.04.007},
       URL = {https://doi-org.wwwproxy1.library.unsw.edu.au/10.1016/j.indag.2022.04.007},
}

@article {Ford,
    AUTHOR = {Ford, K.},
     TITLE = {Vinogradov's integral and bounds for the {R}iemann zeta
              function},
   JOURNAL = {Proc. London Math. Soc. (3)},
  FJOURNAL = {Proceedings of the London Mathematical Society. Third Series},
    VOLUME = {85},
      YEAR = {2002},
    NUMBER = {3},
     PAGES = {565--633},
      ISSN = {0024-6115,1460-244X},
   MRCLASS = {11L15 (11D72 11M06 11M26 11M35 11Y35)},
  MRNUMBER = {1936814},
MRREVIEWER = {D.\ R.\ Heath-Brown},
       DOI = {10.1112/S0024611502013655},
       URL = {https://doi-org.wwwproxy1.library.unsw.edu.au/10.1112/S0024611502013655},
}

@article {Tru,
    AUTHOR = {Trudgian, T.},
     TITLE = {A new upper bound for {$|\zeta(1+it)|$}},
   JOURNAL = {Bull. Aust. Math. Soc.},
  FJOURNAL = {Bulletin of the Australian Mathematical Society},
    VOLUME = {89},
      YEAR = {2014},
    NUMBER = {2},
     PAGES = {259--264},
      ISSN = {0004-9727,1755-1633},
   MRCLASS = {11M06},
  MRNUMBER = {3182661},
MRREVIEWER = {St\'ephane\ R.\ Louboutin},
       DOI = {10.1017/S0004972713000415},
       URL = {https://doi-org.wwwproxy1.library.unsw.edu.au/10.1017/S0004972713000415},
}

@article {Back,
    AUTHOR = {Backlund, R. J.},
     TITLE = {{Sur les zéros de la fonction $\zeta(s)$ de Riemann}},
   JOURNAL = {C. R. Math. Acad. Sci. Paris},
    VOLUME = {158},
      YEAR = {1914},
     PAGES = {1979--1982},
}

@article {Rich,
    AUTHOR = {Richert, H.-E.},
     TITLE = {Zur {A}bsch\"atzung der {R}iemannschen {Z}etafunktion in der
              {N}\"ahe der {V}ertikalen {$\sigma =1$}},
   JOURNAL = {Math. Ann.},
  FJOURNAL = {Mathematische Annalen},
    VOLUME = {169},
      YEAR = {1967},
     PAGES = {97--101},
      ISSN = {0025-5831,1432-1807},
   MRCLASS = {10.41},
  MRNUMBER = {205954},
MRREVIEWER = {S.\ Knapowski},
       DOI = {10.1007/BF01399533},
       URL = {https://doi-org.wwwproxy1.library.unsw.edu.au/10.1007/BF01399533},
}

@article {Loub2015,
    AUTHOR = {Louboutin, S. R.},
     TITLE = {Explicit upper bounds for residues of {D}edekind zeta
              functions},
   JOURNAL = {Mosc. Math. J.},
  FJOURNAL = {Moscow Mathematical Journal},
    VOLUME = {15},
      YEAR = {2015},
    NUMBER = {4},
     PAGES = {727--740},
      ISSN = {1609-3321,1609-4514},
   MRCLASS = {11R42 (11R29)},
  MRNUMBER = {3438830},
MRREVIEWER = {Henri\ Johnston},
       DOI = {10.17323/1609-4514-2015-15-4-727-740},
       URL = {https://doi-org.wwwproxy1.library.unsw.edu.au/10.17323/1609-4514-2015-15-4-727-740},
}

@article {Lou2000,
    AUTHOR = {Louboutin, S.},
     TITLE = {Explicit bounds for residues of {D}edekind zeta functions,
              values of {$L$}-functions at {$s=1$}, and relative class
              numbers},
   JOURNAL = {J. Number Theory},
  FJOURNAL = {Journal of Number Theory},
    VOLUME = {85},
      YEAR = {2000},
    NUMBER = {2},
     PAGES = {263--282},
      ISSN = {0022-314X,1096-1658},
   MRCLASS = {11R42 (11M20)},
  MRNUMBER = {1802716},
       DOI = {10.1006/jnth.2000.2545},
       URL = {https://doi-org.wwwproxy1.library.unsw.edu.au/10.1006/jnth.2000.2545},
}

@article {Sieg,
    AUTHOR = {Siegel, C. L.},
     TITLE = {Absch\"atzung von {E}inheiten},
   JOURNAL = {Nachr. Akad. Wiss. G\"ottingen Math.-Phys. Kl. II},
  FJOURNAL = {Nachrichten der Akademie der Wissenschaften zu G\"ottingen.
              II. Mathematisch-Physikalische Klasse},
    VOLUME = {1969},
      YEAR = {1969},
     PAGES = {71--86},
      ISSN = {0065-5295},
   MRCLASS = {10.65},
  MRNUMBER = {249395},
MRREVIEWER = {S.\ Chowla},
}

@article {Que1,
    AUTHOR = {Qu\^eme, R.},
     TITLE = {Une relation d'in\'egalit\'e{} entre discriminant, nombre de
              classes et r\'egulateur des corps de nombres},
   JOURNAL = {C. R. Acad. Sci. Paris S\'er. I Math.},
  FJOURNAL = {Comptes Rendus des S\'eances de l'Acad\'emie des Sciences.
              S\'erie I. Math\'ematique},
    VOLUME = {306},
      YEAR = {1988},
    NUMBER = {1},
     PAGES = {5--10},
      ISSN = {0249-6291},
   MRCLASS = {11R29},
  MRNUMBER = {929268},
}

@incollection {Que2,
    AUTHOR = {Qu\^eme, R.},
     TITLE = {Relations d'in\'egalit\'es effectives en th\'eorie
              alg\'ebrique des nombres},
 BOOKTITLE = {S\'eminaire de {T}h\'eorie des {N}ombres, 1987--1988
              ({T}alence, 1987--1988)},
     PAGES = {Exp. No. 19, 19},
 PUBLISHER = {Univ. Bordeaux I, Talence},
      YEAR = {1988},
   MRCLASS = {11R29},
  MRNUMBER = {993117},
MRREVIEWER = {Kurt\ Girstmair},
}

@article {LavEgo,
    AUTHOR = {Lavrik, A. F. and Edgorov, \v{Z.}},
     TITLE = {The product of the number of divisor classes by the regulator
              of algebraic number fields},
   JOURNAL = {Izv. Akad. Nauk UzSSR Ser. Fiz.-Mat. Nauk},
  FJOURNAL = {Izvestiya Akademii Nauk UzSSR. Seriya Fiziko-Matematicheskikh
              Nauk. UzSSR Fanlar Akademiyasining Akhboroti.
              Fizika-Matematika Fanlari Seriyasi},
      YEAR = {1975},
    NUMBER = {2},
     PAGES = {77--79},
      ISSN = {0131-8012},
   MRCLASS = {12A50},
  MRNUMBER = {382227},
MRREVIEWER = {J.\ S.\ Joel},
}

@article {Xli,
    AUTHOR = {Li, X.},
     TITLE = {The smallest prime that does not split completely in a number
              field},
   JOURNAL = {Algebra Number Theory},
  FJOURNAL = {Algebra \& Number Theory},
    VOLUME = {6},
      YEAR = {2012},
    NUMBER = {6},
     PAGES = {1061--1096},
      ISSN = {1937-0652,1944-7833},
   MRCLASS = {11N60 (11R42)},
  MRNUMBER = {2968634},
MRREVIEWER = {Ayyadurai\ Sankaranarayanan},
       DOI = {10.2140/ant.2012.6.1061},
       URL = {https://doi-org.wwwproxy1.library.unsw.edu.au/10.2140/ant.2012.6.1061},
}

@article{qingyi,
    author = {Hoo, E. Q. and Teo, L.-P.},
    title = {Explicit bounds of ${|\zeta\left(1+it\right)|}$},
    journal = {Rend. Semin. mat. Univ. Padova},
    year = {2026},
    NUMBER = {6},
    PAGES = {201--233},
    DOI = {10.4171/RSMUP/188},
}

@misc{christensen2026,
    AUTHOR = {Christensen, A. and Pratt, K.},
     TITLE = {Utilizing smoothing techniques to bound
              {$|\zeta(1+it)|$}},
      YEAR = {2026},
      NOTE = {Preprint, arXiv:2607.01424},
       URL = {https://arxiv.org/abs/2607.01424},
}

@article {Land1line,
    AUTHOR = {Landau, E.},
     TITLE = {Neuer {B}eweis des {P}rimzahlsatzes und {B}eweis des
              {P}rimidealsatzes},
   JOURNAL = {Math. Ann.},
  FJOURNAL = {Mathematische Annalen},
    VOLUME = {56},
      YEAR = {1903},
    NUMBER = {4},
     PAGES = {645--670},
      ISSN = {0025-5831,1432-1807},
   MRCLASS = {99-04},
  MRNUMBER = {1511191},
       DOI = {10.1007/BF01444310},
       URL = {https://doi-org.wwwproxy1.library.unsw.edu.au/10.1007/BF01444310},
}

@misc{Fehn,
    AUTHOR = {Fehnker, A.},
     TITLE = {Explicit counting of ideals in number fields
              of arbitrary degree},
      YEAR = {2026},
      NOTE = {Preprint, arXiv:2604.19681},
       URL = {https://arxiv.org/abs/2604.19681},
}

@article {HuWa,
    AUTHOR = {Huxley, M. N. and Watt, N.},
     TITLE = {The number of ideals in a quadratic field. {II}},
   JOURNAL = {Israel J. Math.},
  FJOURNAL = {Israel Journal of Mathematics},
    VOLUME = {120},
      YEAR = {2000},
     PAGES = {125--153},
      ISSN = {0021-2172,1565-8511},
   MRCLASS = {11R47 (11L07 11L40 11P21 11R11)},
  MRNUMBER = {1815373},
MRREVIEWER = {D.\ R.\ Heath-Brown},
       DOI = {10.1007/s11856-000-1274-x},
       URL = {https://doi-org.wwwproxy1.library.unsw.edu.au/10.1007/s11856-000-1274-x},
}

@article {Mul,
    AUTHOR = {M\"uller, W.},
     TITLE = {On the distribution of ideals in cubic number fields},
   JOURNAL = {Monatsh. Math.},
  FJOURNAL = {Monatshefte f\"ur Mathematik},
    VOLUME = {106},
      YEAR = {1988},
    NUMBER = {3},
     PAGES = {211--219},
      ISSN = {0026-9255,1436-5081},
   MRCLASS = {11R44 (11N25)},
  MRNUMBER = {971924},
MRREVIEWER = {Grigori\ Kolesnik},
       DOI = {10.1007/BF01318682},
       URL = {https://doi-org.wwwproxy1.library.unsw.edu.au/10.1007/BF01318682},
}

@article {Now,
    AUTHOR = {Nowak, W. G.},
     TITLE = {On the distribution of integer ideals in algebraic number
              fields},
   JOURNAL = {Math. Nachr.},
  FJOURNAL = {Mathematische Nachrichten},
    VOLUME = {161},
      YEAR = {1993},
     PAGES = {59--74},
      ISSN = {0025-584X,1522-2616},
   MRCLASS = {11R44 (11N45)},
  MRNUMBER = {1251010},
MRREVIEWER = {Ekkehard\ Kr\"atzel},
       DOI = {10.1002/mana.19931610107},
       URL = {https://doi-org.wwwproxy1.library.unsw.edu.au/10.1002/mana.19931610107},
}

@article {Lao,
    AUTHOR = {Lao, H.},
     TITLE = {On the distribution of integral ideals and {H}ecke
              {G}r\"ossencharacters},
   JOURNAL = {Chinese Ann. Math. Ser. B},
  FJOURNAL = {Chinese Annals of Mathematics. Series B},
    VOLUME = {31},
      YEAR = {2010},
    NUMBER = {3},
     PAGES = {385--392},
      ISSN = {0252-9599,1860-6261},
   MRCLASS = {11N37 (11R44)},
  MRNUMBER = {2652933},
MRREVIEWER = {Brundaban\ Sahu},
       DOI = {10.1007/s11401-008-0450-x},
       URL = {https://doi-org.wwwproxy1.library.unsw.edu.au/10.1007/s11401-008-0450-x},
}

@article {Fom,
    AUTHOR = {Fomenko, O. M.},
     TITLE = {On the {D}edekind zeta function. {II}},
   JOURNAL = {J. Math. Sci. (N.Y.)},
  FJOURNAL = {Journal of Mathematical Sciences (New York)},
    VOLUME = {207},
      YEAR = {2015},
    NUMBER = {6},
     PAGES = {923--933},
      ISSN = {1072-3374,1573-8795},
   MRCLASS = {11R42},
  MRNUMBER = {3374220},
MRREVIEWER = {Goran\ Djankovi\'c},
       DOI = {10.1007/s10958-015-2415-4},
       URL = {https://doi-org.wwwproxy1.library.unsw.edu.au/10.1007/s10958-015-2415-4},
}

@article {BellWong,
    AUTHOR = {Bellotti, C. and Wong, P.-J.},
     TITLE = {Improved estimates for the argument and zero-counting
              function of the {R}iemann zeta-function},
   JOURNAL = {Math. Comp.},
    YEAR = {2026},
     NOTE = {Published electronically November 14, 2025},
      DOI = {10.1090/mcom/4133},
      URL = {https://doi.org/10.1090/mcom/4133},
}

@article {HanShenWong,
    AUTHOR = {Hasanalizade, E. and Shen, Q. and Wong, P.-J.},
     TITLE = {Counting zeros of {D}edekind zeta functions},
   JOURNAL = {Math. Comp.},
  FJOURNAL = {Mathematics of Computation},
    VOLUME = {91},
      YEAR = {2021},
    NUMBER = {333},
     PAGES = {277--293},
      ISSN = {0025-5718,1088-6842},
   MRCLASS = {11R42 (11M26)},
  MRNUMBER = {4350540},
MRREVIEWER = {St\'ephane\ R.\ Louboutin},
       DOI = {10.1090/mcom/3665},
       URL = {https://doi.org/10.1090/mcom/3665},
}

@book {LangBook,
    AUTHOR = {Lang, S.},
     TITLE = {Algebraic Number Theory},
 PUBLISHER = {Addison-Wesley Publishing Co., Inc., Reading, Mass.-London-Don
              Mills, Ont.},
      YEAR = {1970},
}

@article {Tru14,
    AUTHOR = {Trudgian, T. S.},
     TITLE = {An improved upper bound for the argument of the {R}iemann
              zeta-function on the critical line {II}},
   JOURNAL = {J. Number Theory},
  FJOURNAL = {Journal of Number Theory},
    VOLUME = {134},
      YEAR = {2014},
     PAGES = {280--292},
      ISSN = {0022-314X,1096-1658},
   MRCLASS = {11M06 (11M26)},
  MRNUMBER = {3111568},
MRREVIEWER = {Haseo\ Ki},
       DOI = {10.1016/j.jnt.2013.07.017},
       URL = {https://doi.org/10.1016/j.jnt.2013.07.017},
}

@book {Neuk,
    AUTHOR = {Neukirch, J.},
     TITLE = {Algebraic number theory},
    SERIES = {Grundlehren der mathematischen Wissenschaften [Fundamental
              Principles of Mathematical Sciences]},
    VOLUME = {322},
 PUBLISHER = {Springer-Verlag, Berlin},
      YEAR = {1999},
     PAGES = {xviii+571},
      ISBN = {3-540-65399-6},
   MRCLASS = {11Rxx (11-02 11S15 11S31 14C40)},
  MRNUMBER = {1697859},
MRREVIEWER = {Cornelius\ Greither},
       DOI = {10.1007/978-3-662-03983-0},
       URL = {https://doi.org/10.1007/978-3-662-03983-0},
}
\end{document}